\documentclass{nyjm}
\pdfoutput=1

\usepackage{graphicx}
\usepackage{epstopdf}
\usepackage{amsmath}
\usepackage{enumerate}
\usepackage{tikz}
\usepackage{wrapfig}
\usepackage{pinlabel}
\usepackage{subfig}
\usepackage{mathabx}
\usepackage{setspace}
\usepackage{hyperref}

\newtheorem{theorem}{Theorem}[section]
\newtheorem{lemma}[theorem]{Lemma}
\newtheorem{corollary}[theorem]{Corollary}
\newtheorem{proposition}[theorem]{Proposition}

\theoremstyle{definition}

\theoremstyle{remark}

\numberwithin{equation}{section}

\newcommand{\T}{\mathcal{T}}
\newcommand{\storus}{S^1 \times D^2}

\newcommand{\laurent}{\mathbb{Z}[A,A^{-1}]}

\newcommand{\G}{\mathcal{G}}
\newcommand{\A}{\mathcal{A}}

\newcommand{\f}{\mathcal{F}}
\newcommand{\D}{\mathcal{D}}
\newcommand{\e}{\mathcal{H}}
\newcommand{\vareps}{\varepsilon}
\newcommand{\ds}{\displaystyle}

\newcommand{\ev}[1]{I_#1^\text{even}}
\newcommand{\od}[1]{I_#1^\text{odd}}
\DeclareMathOperator{\Tet}{Tet}

\begin{document}

\title[Even and odd Kauffman bracket ideals]{Even and odd Kauffman bracket ideals for genus-1 tangles}

\author{Susan M. Abernathy}
\address{Department of Mathematics\\
Angelo State University\\
ASU Station \#10900\\
San Angelo, TX 76909\\
USA}
\email{susan.abernathy@angelo.edu}
\thanks{The first author was supported as a research  assistant by  NSF-DMS-1311911}
\urladdr{http://www.angelo.edu/faculty/sabernathy/}

\author{Patrick M. Gilmer}
\address{Department of Mathematics\\
Louisiana State University\\
Baton Rouge, LA 70803\\
USA}
\email{gilmer@math.lsu.edu}
\thanks{The second author was partially supported by  NSF-DMS-1311911}
\urladdr{www.math.lsu.edu/\textasciitilde gilmer/}

\date{August 22, 2016}


\begin{abstract}

This paper refines previous work by the first author.  We study the question of which links in the 3-sphere can be obtained as closures of a given 1-manifold in 
an unknotted
solid torus
in the 3-sphere
(or genus-1 tangle) 
by adjoining another 1-manifold in the complementary solid torus.  We distinguish between even and odd closures, and define even and odd versions of the Kauffman bracket ideal. These even and odd Kauffman bracket ideals are used to obstruct even and odd tangle closures. 
Using a basis of Habiro's for the even Kauffman bracket skein module of the solid torus, we define bases for the even and odd skein module of the solid torus relative to two points.
These even and odd bases allow us
to compute a finite list of generators for the even and odd Kauffman bracket ideals of a genus-1 tangle.  We do this explicitly for three examples.  Furthermore, we use the even and odd Kauffman bracket ideals to conclude in some cases that the determinants of all even/odd closures of a genus-1 tangle possess a certain divisibility.

\end{abstract}



\maketitle

\tableofcontents

\section{Introduction}\label{intro}

Let $M\subseteq S^3$ be a compact, oriented 3-manifold with boundary. Then an $(M,2n)$-tangle is 1-manifold with $2n$ boundary components properly embedded in $M$.  We refer to $(\storus,2)$-tangles  
where $\storus$ is a unknotted solid torus in $S^3$ 
as genus-1 tangles.

An $(M,2n)$-tangle $\T$ embeds in a link $L\subseteq S^3$ if there exists a complementary 1-manifold $\T^\prime$ with $2n$ boundary components in $S^3-Int(M)$ such that upon gluing $\T^\prime$ to $\T$ along their boundaries, we obtain a link isotopic to $L$.  Such a link is called a closure of $\T$. We refer to $\T^\prime$ as the complementary 1-manifold of the closure.  Note that if $\T$ is a genus-1 tangle, then $\T^\prime$ is also a genus-1 tangle.   The focus of this paper is genus-1 tangle embedding.

In \cite{ab, thesis}, the first author defined  the notion of even and odd closures for any genus-1 tangle $\G$ 
with respect to a longitude $l$ on the boundary of the solid torus which misses the boundary points of $\T$. 
 If we choose $l$ to be the longitude pictured in Figure \ref{fig:krebes}  
 and assume that the boundary points of $\T$ are in the complement of $l$,  then we may think of even and odd closures 
 intuitively as follows.  Even (respectively, odd) closures are those whose complementary 1-manifold passes through the hole of the solid torus containing $\G$ an even (respectively, odd) number of times.  For the remainder of this paper, when we discuss even 
and odd closures, we mean even and odd with respect to the longitude $l$.

In \cite{abkb,thesis}, the first author defined the Kauffman bracket ideal of an $(M,2n)$-tangle $\T$  to be the ideal $I_\T$ of $\laurent$ generated by the reduced Kauffman bracket polynomials of all closures of $\T$.
 This ideal gave an obstruction to embedding.  
In the case $(M,2n)= (B^3,4)$, this ideal was first studied by Przytycki,  Silver and  Williams \cite{psw}.

The first author outlined a method for computing this ideal in the case of genus-1 tangles 
 using skein theory techniques.  In this paper, we define an even and odd version of the   Kauffman bracket ideal for genus-1 tangles.  The even Kauffman bracket ideal of a genus-1 tangle $\G$ is the ideal $\ev{\G}$ generated by the reduced Kauffman bracket polynomials of all even closures of $\G$.  The odd Kauffman bracket ideal $\od{\G}$ is defined similarly.  If an ideal is equal to $\laurent$, we refer to that ideal as trivial.

The following proposition is an immediate consequence of these definitions.

\begin{proposition}
Let $\G$ be a genus-1 tangle.  If $\ev{\G}$ (respectively, $\od{\G}$) is non-trivial, then the unknot is not an even (respectively, odd) closure of $\G$.  More generally, if $L$ is an even (respectively, odd) closure of $\G$, then the  reduced Kauffman bracket polynomial of $L$ must lie in $\ev{\G}$ (respectively, $\od{\G}$).
Finally, we have that $I_\G = \ev{\G} + \od{\G}$. 
\end{proposition}

In \S\ref{section:kbsm}, we recall the basics of  Kauffman bracket skein modules.
In \S\ref{section:evenodd}, we define bases for the even and odd Kauffman bracket skein modules of $\storus$ relative to two points. These even and odd bases are defined  in terms  of a basis for the even skein module of the solid torus due to Habiro \cite{hab}.

In \S\ref{section:graphbasis}, we recall
the graph basis defined in \cite{abkb,thesis}.    In \S\ref{section:method}, we outline a method for computing a finite list of generators for the even and odd Kauffman bracket ideals
 of genus-1 tangles  with two boundary points.

 We note that if the ordinary Kauffman bracket ideal is non-trivial, then both the even and odd Kauffman bracket ideals must be non-trivial.  
 However, the converse is not true. In \S\ref{section:examples}, 
 we examine some specific examples. 
We
  show that Krebes's, tangle $\A$  \cite{kr},   pictured in Figure \ref{fig:krebes}, has trivial even ideal and non-trivial odd ideal. In \cite{abkb,thesis}, we showed that the ordinary Kauffman bracket ideal of Krebes's tangle $\A$ is trivial. 
 We give an example  of a rather simple genus-1 tangle, $\D$ in Figure \ref{fig:smallex}   which has non-trivial even but trivial odd Kauffman bracket ideals. See Figure \ref{fig:smalltangletrivialclosure} for an odd closure 
  of $\D$ 
  which is trivial.
 We 
  also consider a 
  particularly interesting tangle $\e$ (in Figure \ref{fig:85ex}).
 This tangle has non-trivial even ideal and non-trivial odd ideal. Thus it does not posses a a trivial closure, but    
 the ordinary Kauffman bracket ideal of $\e$ is trivial.

  The determinant $\det(L)$ of a link $L$ is a classical link with well-known alternative definitions. On the one hand, this invariant is the absolute value of the determinant of a Seifert matrix for $L$ symmetrized. It can also be described as the order of the first homology group of the double branched cover of $S^3$ along $L$ (this is interpreted to be zero if this homology group is infinite). In \cite{ab}, the first author used the homology of double branched covers to show that any odd closure of Krebes's tangle has determinant divisible by $3$. Here  we can reach this result as a consequence of our calculation of the odd Kauffman 
 bracket ideal Krebes's tangle. 
We also obtain similar results for other tangles in the same way. Ultimately
this approach to
the determinants of closures 
rests
  on Jones' observation \cite[Corollary 13]{J} that his polynomial evaluated at $t=-1$ is the determinant (up to sign), and Kauffman's bracket 
  polynomial
  description \cite{K} of the Jones polynomial.
In  \S\ref{section:determinant}, we relate the even and odd ideals of an genus-1 tangle to the determinants of even and odd closures of that tangle.

\begin{figure}\labellist
\small\hair 2pt
\endlabellist
\includegraphics[height=1in]{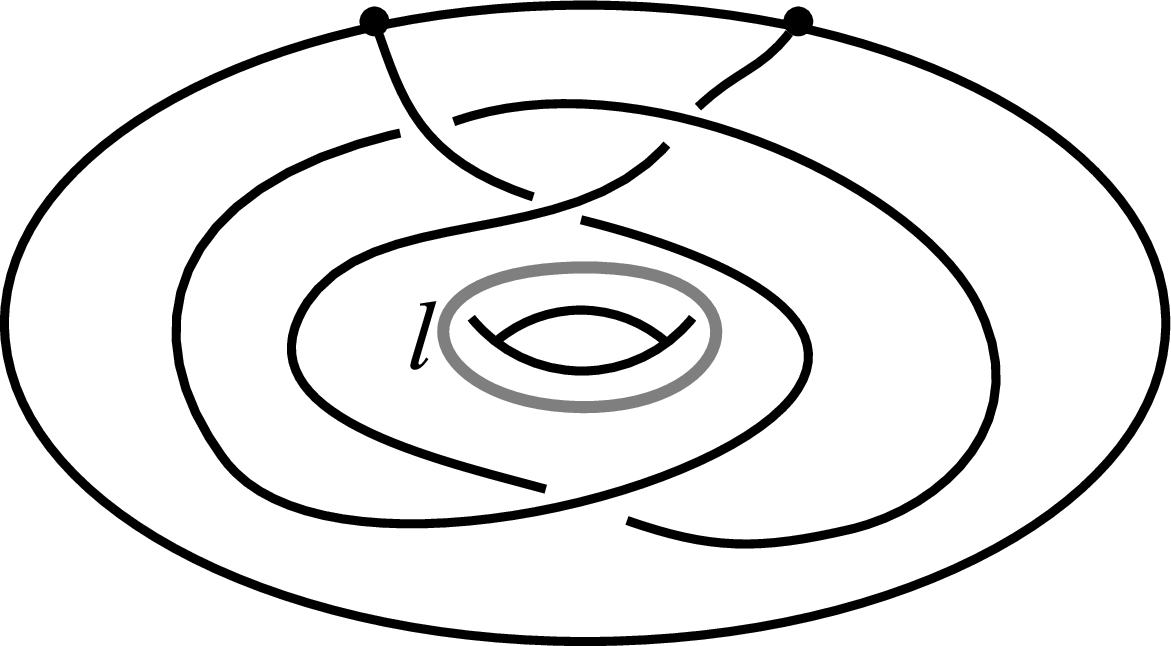}\caption{Krebes's tangle $\A$  and a longitude $l$.}
\label{fig:krebes}
\end{figure}


\section{Kauffman bracket skein modules}\label{section:kbsm}

The Kauffman bracket polynomial of a framed link $D$, denoted by $\langle D \rangle$, is an element of $\laurent$ given by the following three relations, where $\delta = -A^2-A^{-2}$:
\begin{enumerate}[(i)]
\item $\displaystyle \langle\begin{minipage}{.5in}\begin{center}\includegraphics[width=.4in]{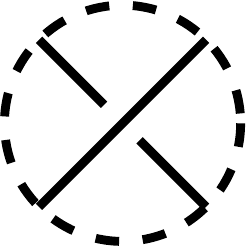}\end{center}\end{minipage} \rangle= A \langle\begin{minipage}{.5in}\begin{center}\includegraphics[width=.4in]{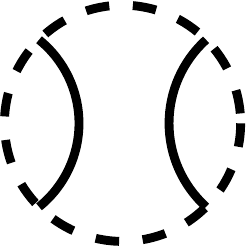}\end{center}\end{minipage} \rangle+ A^{-1}\langle\begin{minipage}{.5in}\begin{center}\includegraphics[width=.4in]{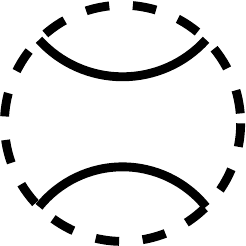}\end{center}\end{minipage} \rangle   $
\item $ \langle D^\prime \coprod \begin{minipage}{.4in}\begin{center}\includegraphics[width=.3in]{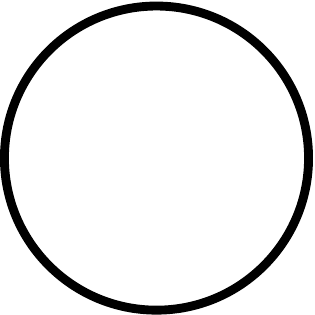}\end{center}\end{minipage} \rangle= \delta \langle D^\prime\rangle.$
\item $ \langle \begin{minipage}{.1in}\begin{center}\end{center}\end{minipage} \rangle= 1.$
\end{enumerate}
We let $\langle D \rangle ^\prime$ denote the reduced Kauffman bracket polynomial of $D$; that is, 
where 
$\langle D \rangle ^\prime = \langle D\rangle/ \delta$.

The Kauffman bracket skein module of a 3-manifold $M$ is the $\laurent$-module $K(M)$ generated by isotopy classes of framed links in $M$ modulo the Kauffman bracket relations above.

Of particular concern to us is the relative Kauffman bracket skein module.  Let $M$ be a compact oriented 3-manifold with boundary and a set of $m$ specified marked 
framed
 points on $\partial M$.  Then the Kauffman bracket skein module of $M$ relative to the $m$ marked points is the $\laurent$-module $K(M,m)$ generated by isotopy classes of framed 1-manifolds  with boundary the marked 
framed
 points modulo the above Kauffman bracket relations.  We can view any genus-1 tangle (equipped with the blackboard framing) as a skein element in $K(\storus,2)$.

As in \cite{abkb, thesis}, we generalize the Hopf pairing on $K(\storus)$ defined in \cite{bhmv92} to obtain the relative Hopf pairing $\langle\text{ , }\rangle: K(\storus,2) \times K(\storus,2) \rightarrow K(S^3) = \laurent$.  Given $a$ and $b$ in $K(\storus,2)$, we let $$\langle a, b \rangle =  \left\langle\hskip.05in\begin{minipage}{1in}\includegraphics[width=1in]{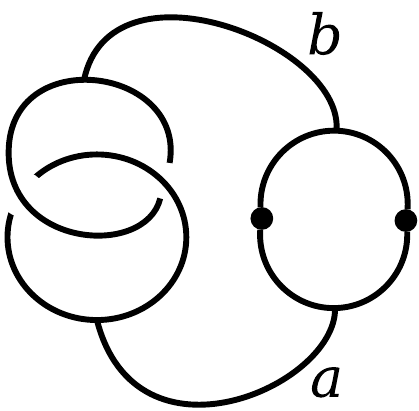}\end{minipage}\hspace{.04in}\right\rangle$$ where $a$ and $b$ lie in regular neighborhoods of the trivalent graphs.

If a genus-1 tangle $\G$ embeds in a link $L\subseteq S^3$, then we can describe this tangle embedding using the relative Hopf pairing.  We have that $\langle L \rangle = \langle\G,\G^\prime\rangle$ for some $\G^\prime \in K(\storus,2)$.


\section{Even and odd relative skein modules}\label{section:evenodd}

As in \cite{bhmv92},   we let $z$ denote a standard banded core  of $\storus$ and the element this core represents in  $K(\storus)$.  
A basis for  $K(\storus)$ is given by $\{z^n\}_{n \ge 0}$. 
As described in \cite{hab}, one can obtain a $\mathbb{Z}_2$-graded algebra structure on the Kauffman bracket skein module $K(\storus)$ by letting $K^\text{even}(\storus)$ be the subalgebra of $K(\storus)$ generated by $z^2$ and $K^\text{odd}(\storus)$ be $zK^\text{even}(\storus)$. Then, 
one has
that $K(\storus) = K^\text{even}(\storus) \oplus K^\text{odd}(\storus)$. This is because  the Kauffman skein relations respect $\mathbb{Z}_2$-homology classes   \cite[p.105]{gh}.

Suppose now that $\storus$ is equipped with two marked framed points in $\partial(\storus)$ and an  essential curve $l$ in $\partial(\storus)$  missing the marked points and which bounds a disk $\mathfrak D$ in $\storus$. 
 Let  $u$ be a framed 1-manifold in $\storus$ with the two given marked points as boundary. Then we say that $u$ is even (respectively, odd), if $u$ intersects $\mathfrak D$ an even (respectively, odd) number of times.  
Let $K^\text{even}(\storus,2)$ and $K^\text{odd}(\storus,2)$ be the submodules of $K(\storus,2)$ generated by all even and odd 1-manifolds, respectively.  Then, we have that $K(\storus,2) = K^\text{even}(\storus,2) \oplus K^\text{odd}(\storus,2)$. 
 Note that if $L$ is an even closure of a genus-1 tangle $\G$, then the Kauffman bracket polynomial of $L$ can be written as $\langle L \rangle = \langle \G,\G^\prime\rangle$ where
 $\G \in K(\storus,2)$ and 
  $\G^\prime \in K^\text{even}(\storus,2)$.  Here $l$ is a ``longitude'' for the first copy of  $\storus$ and  a ``meridian'' for the second copy of 
 $\storus$.The analogous statement is true for odd closures. 

\begin{figure}
\includegraphics[height=1.2in]{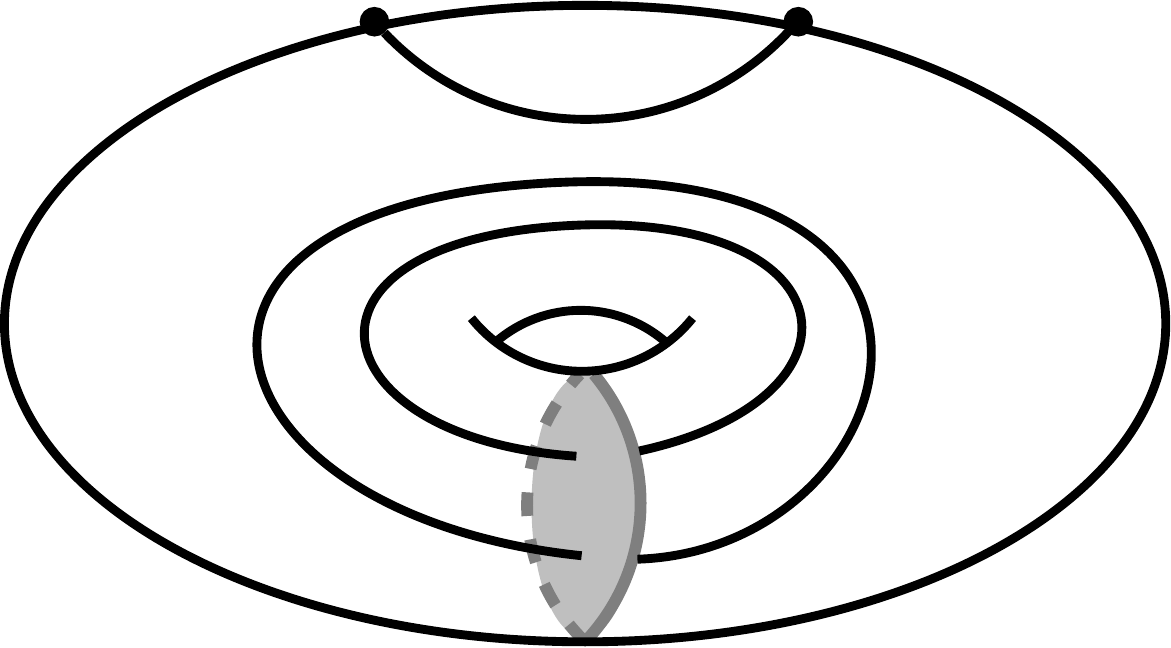}
\caption{An even element of $K(\storus,2)$, where $\mathfrak D$ is the shaded disk}\label{fig:evenskeinelt}
\end{figure}

In \cite{bhmv92},  a basis $\{Q_n\}_{n \ge 0}$ for $K(\storus)$ is given. It is orthogonal  with respect to the Hopf pairing. Here $Q_n =\ds \prod_{i=0}^{n-1} (z-\phi_i)$, where $\phi_i = -A^{2i+2}-A^{-2i-2}$ (in \cite{bhmv92} and elsewhere this is denoted $\lambda_i$).   
In the case $n=0$, we interpret the empty product  as the identity which is represented by the empty link.  
In \cite{hab}, Habiro modifies the definition of $Q_n$ 
 to obtain a new basis for the even submodule $K^\text{even}(\storus)$ given by $S_n = \ds\prod_{i=0}^{n-1} (z^2-\phi_i^2)$ for $n\geq 0$.  Note that $S_n = Q_n\ds \prod_{i=0}^{n-1}(z + \phi_i)$.

We adapt Habiro's basis to obtain bases for the $K^\text{even}(\storus,2)$ and $K^\text{odd}(\storus,2)$.  We refer to them as the even basis and odd basis, respectively, and 
define them as follows
 (using the same $\mathfrak D$ as pictured in Figure \ref{fig:evenskeinelt}).
The even basis consists of the following elements, where $n\geq 0$: $$x_n^\text{even} =  \begin{minipage}{1in}\includegraphics[width=1in]{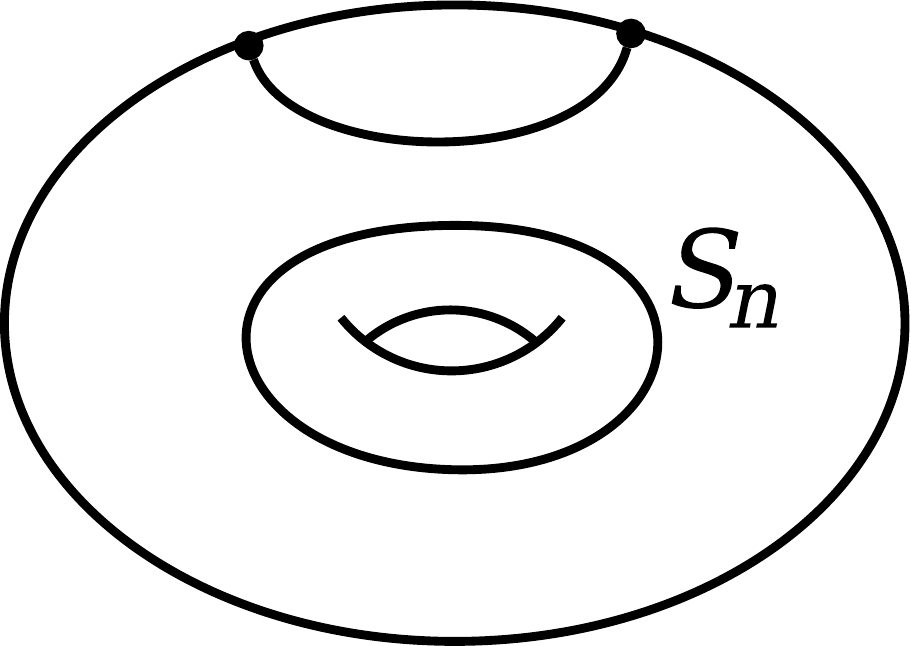}\end{minipage}\text{ and }y_n^\text{even} =  \begin{minipage}{1.03in}\includegraphics[width=1in]{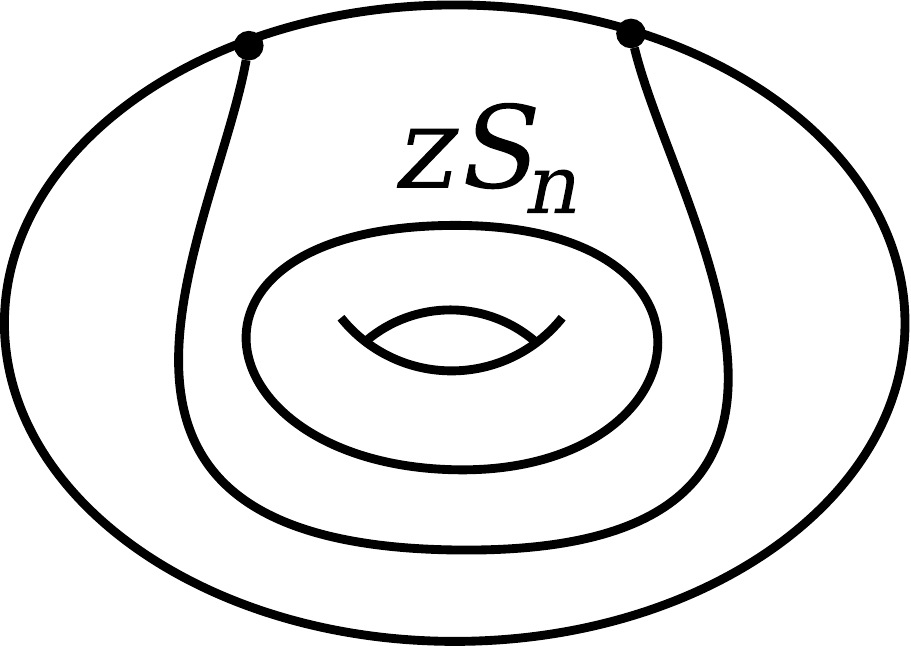}\end{minipage}.$$  Similarly, the odd basis consists of the following elements, where $n\geq 0$: $$x_n^\text{odd} =  \begin{minipage}{1in}\includegraphics[width=1in]{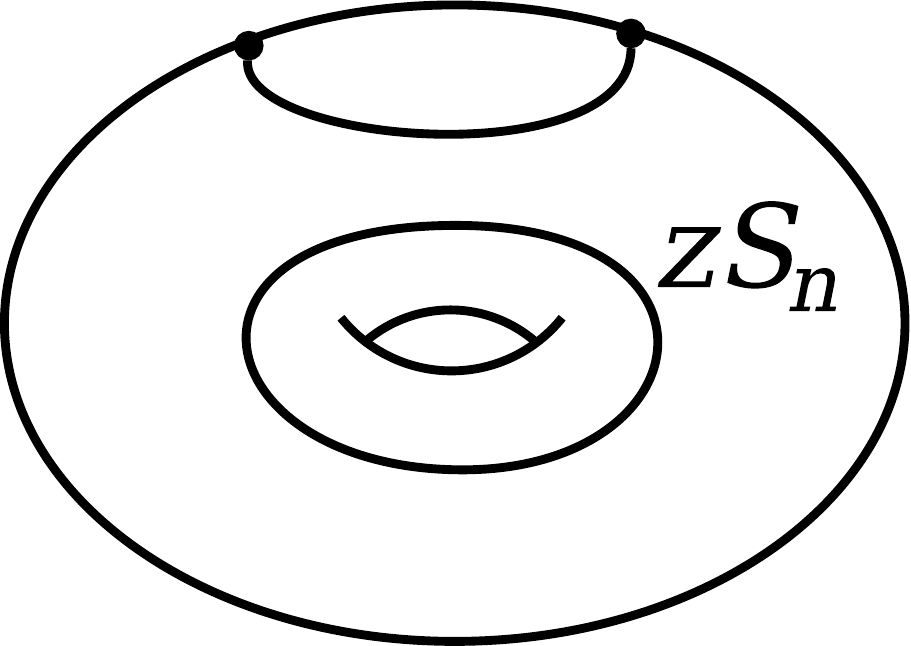}\end{minipage}\text{ and }y_n^\text{odd} =  \begin{minipage}{1.03in}\includegraphics[width=1in]{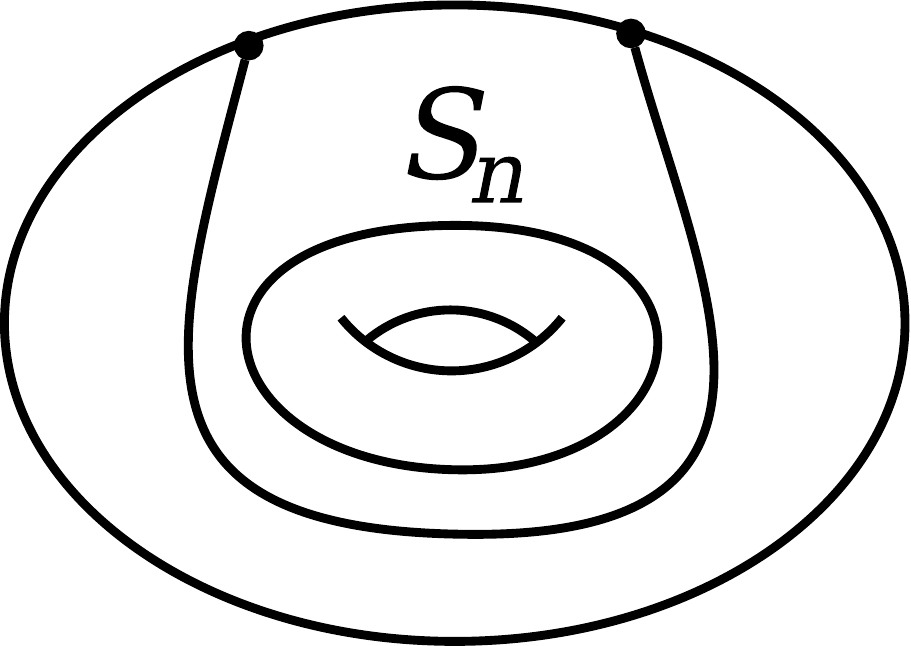}\end{minipage}.$$

That these are bases follows ultimately from the basis for $K(\storus,2)$ consisting of framed links described by isotopy classes of diagrams without crossings and without contractible  loops in $S^1 \times D^2$. One also uses the fact that there is a triangular  unimodular change of basis matrix over $\laurent$ relating the bases $\{S_n\}$ and \{$z^{2m} \}$ for $K^\text{even}(\storus)$.


\section{Graph basis of $K_R(\storus,2)$}\label{section:graphbasis}

Trivalent graphs will be interpreted as in \cite{abkb, thesis, gh, kl, mv,l}.
  Any unlabelled edge is assumed to be colored one.
The colors of the three edges incident to a single vertex must form an admissible triple.  Given non-negative integers $a$, $b$, and $c$, the triple $(a,b,c)$ is admissible if $|a-b| \leq c\leq a + b$ and $a + b + c \equiv 0 (\text{mod } 2)$.  We use the notation of  \cite{kl}: $\Delta_n$, $\theta (a,b,c)$, $\Tet \begin{bmatrix} a & b & e\\ c & d & f  \end{bmatrix}$, and $\lambda^{a \text{ }b}_c$.

We use the graph basis defined in \cite{abkb,thesis}.  Given a pair of non-negative integers $(i,\vareps)$ such that $\vareps = i+1$ or $\vareps = i-1$, let \[g_{i,\vareps} =\hspace{.05in}\begin{minipage}{1.25in}\includegraphics[width=1.25in]{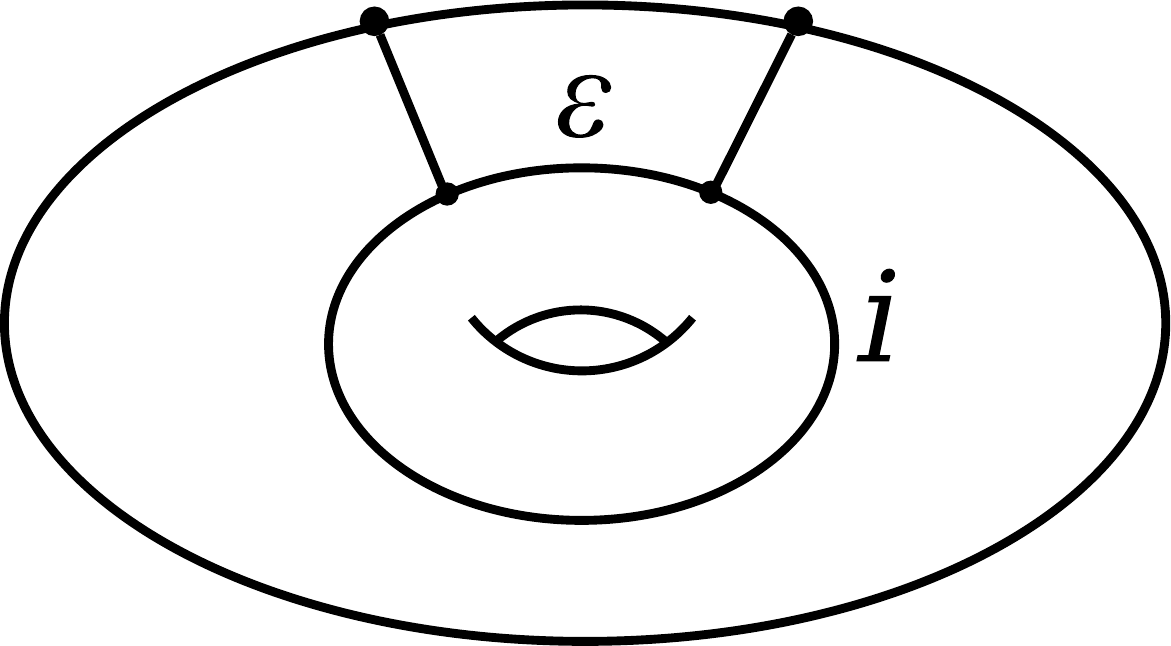}\end{minipage}.\]  

Let $R$ denote the ring $\laurent$ localized by inverting $A^{k}-1$ for all natural numbers $k$, and let $K_R(M,m)$ denote the Kauffman bracket skein module of $M$ relative to $m$ points with coefficients in $R$.  According to \cite[Theorem 2.3]{pr}, we have that $K_R(M,m) = K(M,m) \otimes R$, so we can essentially view $K(M,m)$ as a subset of $K_R(M,m)$.  We make this distinction because when computing a finite list of generators for the even and odd Kauffman bracket ideals, we 
pass through
 $K_R(\storus,2)$ when using the doubling pairing 
defined in \cite[\S2.3]{abkb}. However, each of the generators we obtain is in fact an element of $K(\storus,2)$.

Recall, according to \cite{hp}, $K_R(S^1\times S^2)/\text{torsion}$ is isomorphic to $R$. Let $\psi: K_R(S^1\times S^2) \rightarrow R$ be the epimorphism that sends the empty link to $1 \in R.$
The doubling pairing is defined to be the symmetric pairing $\langle \text{ , }\rangle_D: K_R(\storus,2) \times K_R(\storus,2) \rightarrow R$ obtained by gluing two solid tori containing skein elements 
together via a certain orientation-reversing homeomorphism to obtain a skein element in $S^1\times S^2$, and evaluating this skein element under $\psi.$  Figure \ref{fig:doublingpairing} illustrates the doubling pairing of two graph basis elements. The thick dark colored loop indicates where a $0$-framed surgery is to be performed, converting $S^3$ to $S^1\times S^2$.
According to \cite[Theorem 2.4]{abkb}, the graph basis is orthogonal with respect to the doubling pairing.

\begin{figure}
$$\langle g_{i,\vareps}, g_{i^\prime,\vareps^\prime}\rangle_D =\left\langle \begin{minipage}{1.2in}\includegraphics[height=1in]{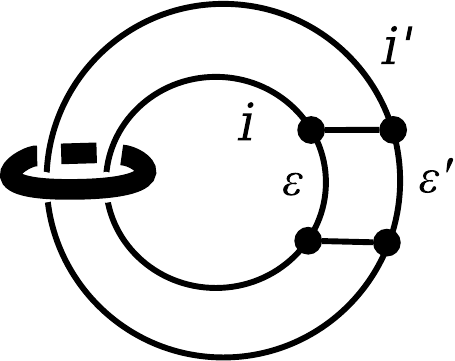}\end{minipage}\right\rangle$$
\caption{The doubling pairing of two graph basis elements. The bold loop indicates a 0-framed surgery.}\label{fig:doublingpairing}
\end{figure}


\section{Applications to genus-1 tangle embedding}\label{section:method}

Let $\G$ be a genus-1 tangle.  
The Kauffman bracket polynomial of any even closure $L$ of $\G$ can be written as $\langle L \rangle = \langle \G,\G^\prime\rangle$ where $\G^\prime \in K^\text{even}(\storus,2)$.
So, $\langle \G,x^\text{even}_n\rangle/\delta$ and $\langle \G,y^\text{even}_n\rangle/\delta$ form a generating set for $I_\G^\text{even}$.
Similarly, $\langle \G,x^\text{odd}_n\rangle/\delta$ and $\langle \G,y^\text{odd}_n\rangle/\delta$ form a  generating set for $I_\G^\text{odd}$.
 We will see in this section that these generating sets are finite.

  We follow the same basic procedure as in \cite{abkb, thesis} to obtain finite lists of generators for $I_\G^\text{even}$ and $I_\G^\text{odd}$. First, we write $\G$ as a linear combination of graph basis elements $\G = \sum c_{i,\vareps}g_{i,\vareps}$.  Since the graph basis is orthogonal, we have that $c_{i,\vareps} = \langle \G, g_{i,\vareps} \rangle_D / \langle  g_{i,\vareps} , g_{i,\vareps} \rangle_D$ and only finitely many $c_{i,\vareps}$ are non-zero.

We then use this linear combination to compute the relative Hopf pairing of $\G$ with the even (respectively, odd) basis to obtain a generating set for $I_\G^\text{even}$ (respectively, $I_\G^\text{odd}$).  The following results allow us to compute the relative Hopf pairing of the graph basis with the even and odd bases.  The proof of Lemma \ref{lemma:removings} below  is similar to that of  \cite[Lemma 4.1]{abkb}.

\begin{lemma}\label{lemma:removings}
$$\begin{minipage}{.75in}\includegraphics[height=.8in]{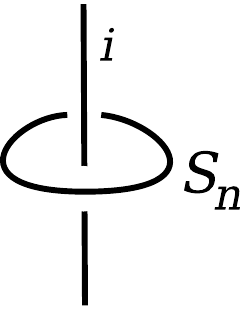}\end{minipage} = \ds\prod_{k=0}^{n-1}(\phi_i^2 - \phi_k^2) \hspace{.04in}\begin{minipage}{.3in}\includegraphics[height=.65in]{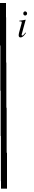}\end{minipage}$$
If $n=0$, we interpret 
$\ds\prod_{k=0}^{n-1}(\phi_i^2 - \phi_k^2)$ as $1$.
\end{lemma}

Notice that $\prod_{k=0}^{n-1}(\phi_i^2 - \phi_k^2)$ is zero if $n > i$.

\begin{proposition}\label{prop:pairing}

\begin{enumerate}[(i)]
\item $\langle g_{i,\vareps},x^\text{even}_n\rangle = \theta(1,i,\vareps) \ds\prod_{k=0}^{n-1}(\phi_i^2 - \phi_k^2)$.\\
\item $\langle g_{i,\vareps},y^\text{even}_n\rangle = \phi_i (\lambda^{i\text{ } 1}_\vareps)^{-2} \theta(1,i,\vareps) \ds\prod_{k=0}^{n-1}(\phi_i^2 - \phi_k^2)$.\\
\item $\langle g_{i,\vareps},x^\text{odd}_n\rangle = \phi_i \theta(1,i,\vareps) \ds\prod_{k=0}^{n-1}(\phi_i^2 - \phi_k^2)$.\\
\item $\langle g_{i,\vareps},y^\text{odd}_n\rangle = (\lambda^{i\text{ } 1}_\vareps)^{-2}  \theta(1,i,\vareps) \ds\prod_{k=0}^{n-1}(\phi_i^2 - \phi_k^2)$.\\
\end{enumerate}
\noindent Each of these is zero 
if
$n>i$.
\end{proposition}

\begin{proof}
\begin{enumerate}[(i)]
\item We have from Lemma \ref{lemma:removings} that $$ \langle g_{i,\vareps}, x^\text{even}_n\rangle = \hspace{.05in}\begin{minipage}{1in}\includegraphics[height=.75in]{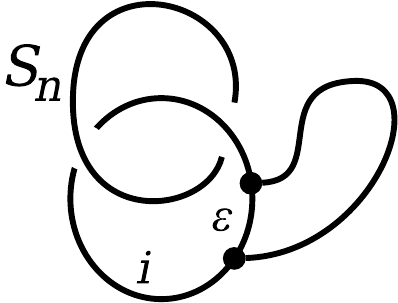}\end{minipage} = \ds\prod_{k=0}^{n-1}(\phi^2_i - \phi^2_k)\hspace{.05in} \begin{minipage}{.65in}\includegraphics[height=.5in]{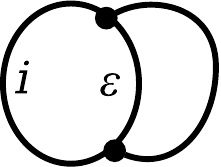}\end{minipage} =  \ds\theta(1,i,\vareps)\prod_{k=0}^{n-1}(\phi^2_i - \phi^2_k).$$\\

\item Using Lemma \ref{lemma:removings}, \cite[ equations 2.2, 2.5]{abkb} and $\lambda^{i\text{ } j}_k= \lambda^{j\text{ }i}_k,$
$$\begin{array}{r c  l}
\langle g_{i,\vareps}, y^\text{even}_n\rangle & = & \begin{minipage}{.75in}\includegraphics[height=.85in]{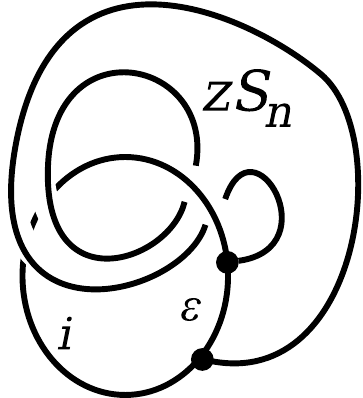}\end{minipage}  = \hspace{.05in}  \ds\phi_i\prod_{k=0}^{n-1}(\phi^2_i - \phi^2_k)\hspace{.05in} \begin{minipage}{.75in}\includegraphics[height=.85in]{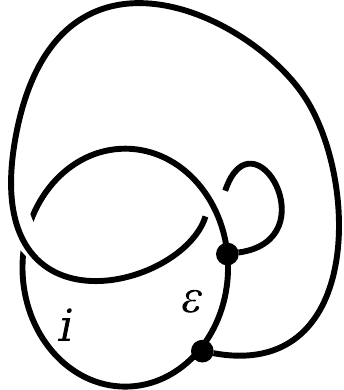}\end{minipage}\\
& & \\
& = & \ds\phi_i(\lambda^{i\text{ } 1}_\vareps)^{-1}(\lambda^{1\text{ } i}_\vareps)^{-1}\prod_{k=0}^{n-1}(\phi^2_i - \phi^2_k)\hspace{.05in} \begin{minipage}{.6in}\includegraphics[height=.5in]{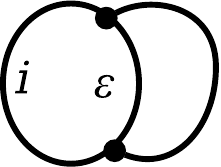}\end{minipage}\\
& = & \ds\phi_i(\lambda^{i\text{ } 1}_\vareps)^{-2}\theta(1,i,\vareps)\prod_{k=0}^{n-1}(\phi^2_i - \phi^2_k).
\end{array}$$

\item 
\begin{align*}\langle g_{i,\vareps}, x^\text{odd}_n\rangle & = \hspace{.05in}\begin{minipage}{1.1in}\includegraphics[height=.75in]{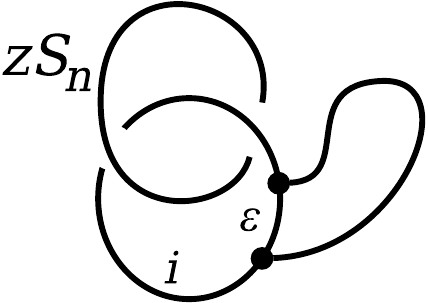}\end{minipage} = \ds\phi_i\prod_{k=0}^{n-1}(\phi^2_i - \phi^2_k)\hspace{.05in} \begin{minipage}{.65in}\includegraphics[height=.5in]{pairinggraphx2.pdf}\end{minipage}\\
& =  \ds\phi_i\theta(1,i,\vareps)\prod_{k=0}^{n-1}(\phi^2_i - \phi^2_k).\end{align*}\\

\item 
$$\begin{array}{r c  l}
\langle g_{i,\vareps}, y^\text{odd}_n\rangle & = & \begin{minipage}{.75in}\includegraphics[height=.85in]{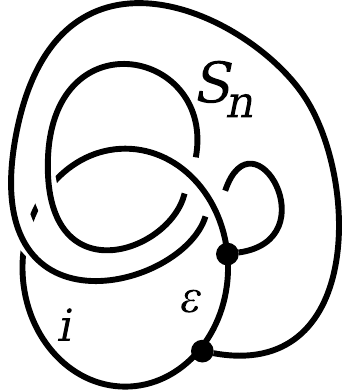}\end{minipage}  = \hspace{.05in}  \ds\prod_{k=0}^{n-1}(\phi^2_i - \phi^2_k)\hspace{.05in} \begin{minipage}{.75in}\includegraphics[height=.85in]{pairinggraphy2.pdf}\end{minipage}\\
& & \\
& = & \ds(\lambda^{i\text{ } 1}_\vareps)^{-1}(\lambda^{1\text{ } i}_\vareps)^{-1}\prod_{k=0}^{n-1}(\phi^2_i - \phi^2_k)\hspace{.05in} \begin{minipage}{.6in}\includegraphics[height=.5in]{pairinggraphy3.pdf}\end{minipage}\\
& = & \ds\ds(\lambda^{i\text{ } 1}_\vareps)^{-2} \theta(1,i,\vareps)\prod_{k=0}^{n-1}(\phi^2_i - \phi^2_k).
\end{array}$$
\end{enumerate}
\end{proof}

 Proposition \ref{prop:pairing} implies that only finitely many of $\langle \G,x^\text{even}_n\rangle/\delta$ and $\langle \G,y^\text{even}_n\rangle/\delta$ will be non-zero.
 Similarly only finitely many of $\langle \G,x^\text{odd}_n\rangle/\delta$ and $\langle \G,y^\text{odd}_n\rangle/\delta$ will be non-zero.  This is why we choose to define the even/odd bases for $K^\text{even/odd}(\storus,2)$ as we did above.          
 If, for instance, we replace $S_n$ by $z^{2n}$ in these definitions, then we would not have this finiteness.


\section{Examples}\label{section:examples}

We compute the even and odd Kauffman bracket ideals for three tangles $\A$, $\D$, and $\e$.  In each of these computations, our first step is to compute the doubling pairing of the tangle in question with the graph basis.  We leave out the full computation for the sake of brevity, but we follow the same procedure as in \cite[Appendix A]{abkb}.
We then write each tangle as a linear combination of graph basis elements. It turns out that due to admissibility conditions, $\A$, $\D$, and $\e$ may all be written as $c_{0,1}g_{0,1} + c_{2,1}g_{2,1} + c_{2,3}g_{2,3}$ for some coefficients $c_{i,\vareps} \in R$. Thus we have using Proposition \ref{prop:pairing}:

\begin{lemma}\label{i012} If $\G$= $\A$, $\D$, or $\e$, 
$I_\G^\text{even}$ is generated by $\langle \G, x^\text{even}_i\rangle/\delta$ and $\langle \G, y^\text{even}_i\rangle/\delta$ where $0\leq i \leq 2$.  Similarly, $I_\G^\text{odd}$ is generated by $\langle \G, x^\text{odd}_i\rangle/\delta$ and $\langle \G, y^\text{odd}_i\rangle/\delta$ where $0\leq i \leq 2$.
\end{lemma}

For  $\G$= $\A$, $\D$ or $\e$, we followed the same procedure as in \cite{abkb,thesis},  to find  $I_\A^\text{even}$ and $I_\A^\text{odd}$, using  Proposition \ref{prop:pairing}, Lemma \ref{i012}, and Mathematica. One can verify directly that the claimed ideals are indeed non-trivial using the computations in \S\ref{section:determinant}.

\subsection{Krebes's tangle $\A$ }\label{subsection:krebes}

We consider the genus-1 tangle given by Krebes \cite{kr} pictured in Figure \ref{fig:krebes}.
We have that
$$ \langle \A, g_{i,\vareps}\rangle_D  =  \begin{minipage}{.11in}\includegraphics[width=1.1in]{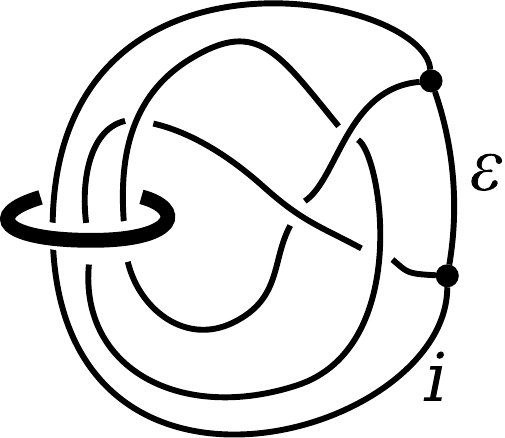}\end{minipage} \hskip1in  \quad \text{is the sum of } $$

\noindent $$ 
 \frac{  \lambda^{1 \text{ }1}_i (\lambda^{1\text{ }1}_j)^{-1} (\lambda^{1\text{ }1}_k)^{-1} (\lambda^{1\text{ }1}_l)^{-1} \Delta_j  \Delta_k  \Delta_l
 \Tet \begin{bmatrix} 1 & i & \vareps\\ 1 & j & 1  \end{bmatrix}
  \Tet  \begin{bmatrix} l & 1 & j \\ 1 & k & 1  \end{bmatrix}
  \Tet \begin{bmatrix} 1 & \vareps & 1 \\ k & l & j  \end{bmatrix}}
 {\theta(1, 1, i) \theta(1,1,j) \theta(1,1,k) \theta(1,1,l) \theta(1,j,\vareps) \theta(\vareps, k, 1)  \theta(l, k, j)}
$$

\noindent 
over all  $j$, $k$, and $l$ such that the following triples are admissible:  $(1,1,i)$, $(1,1,j)$, $(1,j,\vareps)$, $(1,1,k)$, $(\varepsilon,k,1)$, $(1,1,l)$, and $(l,k,j)$.  Admissibility conditions imply that 0 and 2 are the only possible admissible values for $j$, $k$, and $l$. Note that, if $i \ne 0, 2$, there are no such $j$,$k$, $l$, and the given sum is over an empty index set.
Thus, the value of the sum is zero.  
So $\langle \A, g_{i,\vareps}\rangle_D =0$ unless $i=0$ or $i=2$.
 
The coefficients for $\A$ as a linear combination  of the graph basis are
\begin{equation*} c_{0,1} = \frac{-1-A^8+A^{12}}{1+A^4}, \quad c_{2,1} = \frac{-1+A^4+A^{12}}{A^6+A^{10}+A^{14}}, \text{ and } c_{2,3} = 1.\end{equation*}

After further computation, we obtain the generating sets given in the following result.

\begin{proposition}
The even Kauffman bracket ideal $I_\A^\text{even}$ of Krebes's tangle $\A$ is trivial.  The odd Kauffman bracket ideal of $\A$ is $I_\A^\text{odd} = \langle 9, 4+A^4\rangle$ which is non-trivial.
\end{proposition}

\subsection{A small tangle, $\D$}\label{subsection:smallex}

We now consider the genus-1 tangle $\D$ pictured in Figure \ref{fig:smallex}.   In contrast to Krebes's example, $\D$ has a non-trivial even Kauffman bracket ideal and a trivial odd Kauffman bracket ideal.
We remark that we could also obtain a tangle with these properties from Krebes's tangle by sliding one endpoint of Krebes's tangle across the longitude $\ell$ and dragging the rest of tangle along behind this endpoint.

\begin{figure}[h]
\includegraphics[height=1.2in]{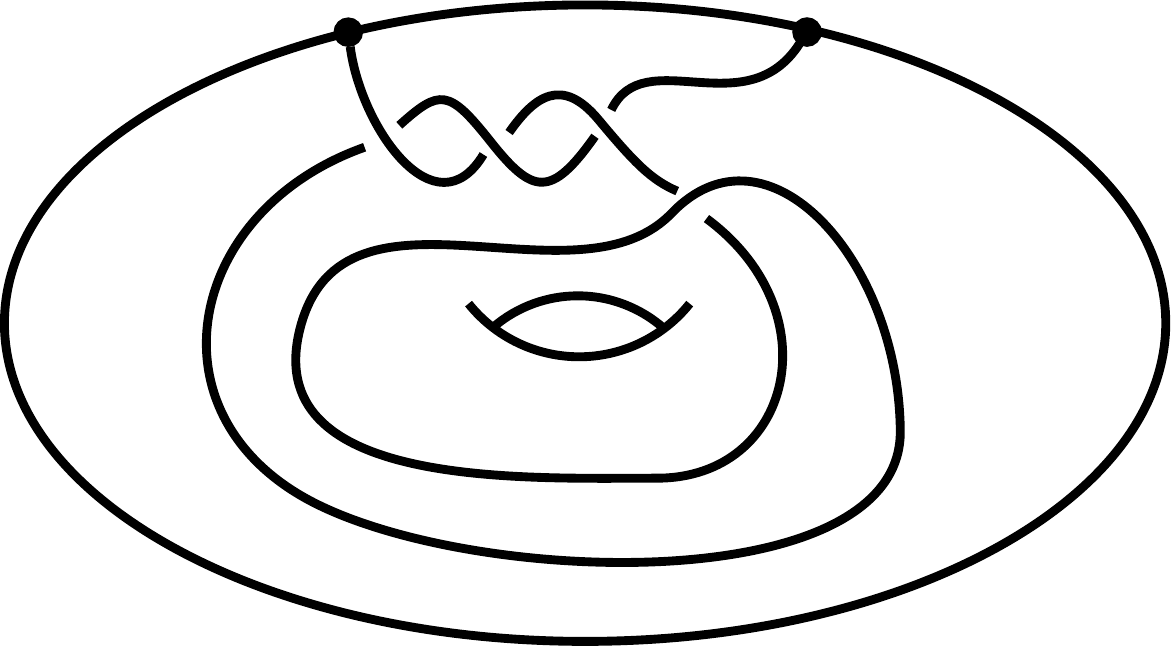}
\caption{A genus-1 tangle, denoted by $\D$.}\label{fig:smallex}
\end{figure}

We have that 
$$ \langle \D, g_{i,\vareps}\rangle_D  =  \begin{minipage}{1.1in}\includegraphics[width=1.1in]{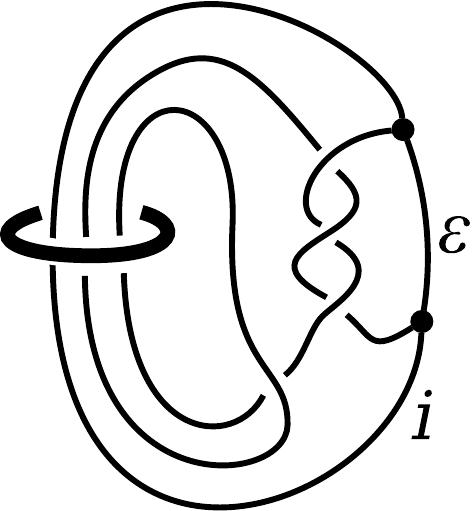}\end{minipage}  \quad \text{is the sum of } $$
$$\frac{  \lambda^{1 \text{ }1}_i (\lambda^{1\text{ }1}_j)^{-3}\Delta_j
\Tet  \begin{bmatrix} 1 & 1 & j \\ 1 & \vareps & i   \end{bmatrix}
\Tet  \begin{bmatrix} 1 & i & \vareps \\ 1 & j & 1   \end{bmatrix}}
{\theta(1, 1, i)\theta(1,1,j)\theta(1,\vareps,j)}
$$

\noindent 
over all integers $j$ such that the following are admissible triples: $(1,1,i)$, $(1,1,j)$, and $(1, \varepsilon,j)$.  
Admissibility conditions imply that 0 and 2 are the only possible admissible values for $j$, and $\langle \D, g_{i,\vareps}\rangle_D =0$ unless $i=0$ or $i=2$.

The coefficients for $\D$ as a linear combination  of the graph basis are
\begin{equation*}
 c_{0,1} = \frac{1-A^4-A^{12}}{A^2+A^6}, \quad c_{2,1} = \frac{1+A^8-A^{12}}{A^8+A^{12}+A^{16}} , 
 \text{ and } c_{2,3} = A^2 .\end{equation*}

We obtain the following generating sets after further computation.

\begin{proposition}
The even Kauffman bracket ideal of $\D$ is $I_\D^\text{even} = \langle 9,-2+A^4\rangle$ which is non-trivial. The odd Kauffman bracket ideal $I_\D^\text{odd}$ of $\D$ is trivial. 
\end{proposition}

Indeed, one can see that the odd closure shown in Figure \ref{fig:smalltangletrivialclosure} is the unknot, so $I_\D^\text{odd}$ must be trivial.

\begin{figure}[h]
\includegraphics[height=1.5in]{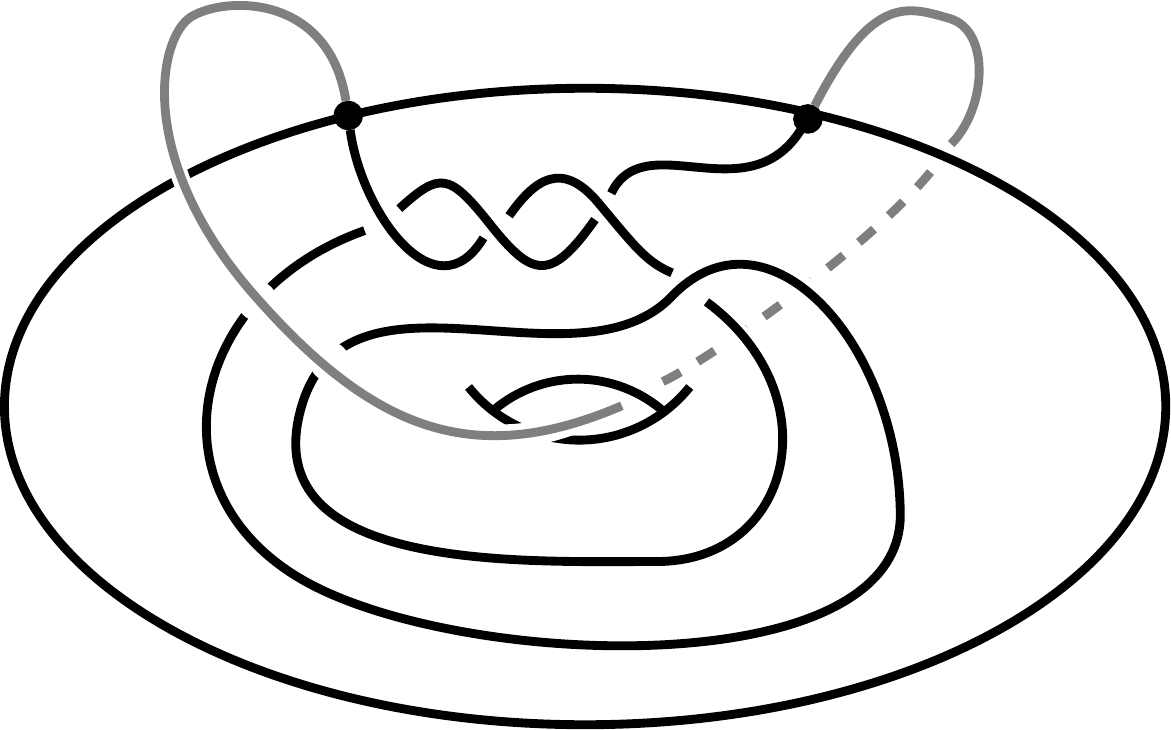}
\caption{A trivial odd closure of the tangle $\D$.}\label{fig:smalltangletrivialclosure}
\end{figure}

\subsection{A particularly interesting tangle, $\e$} \label{subsection:85ex}

We consider the genus-1 tangle $\e$ pictured in Figure \ref{fig:85ex}.  
\begin{figure}[h]
\includegraphics[height=1.4in]{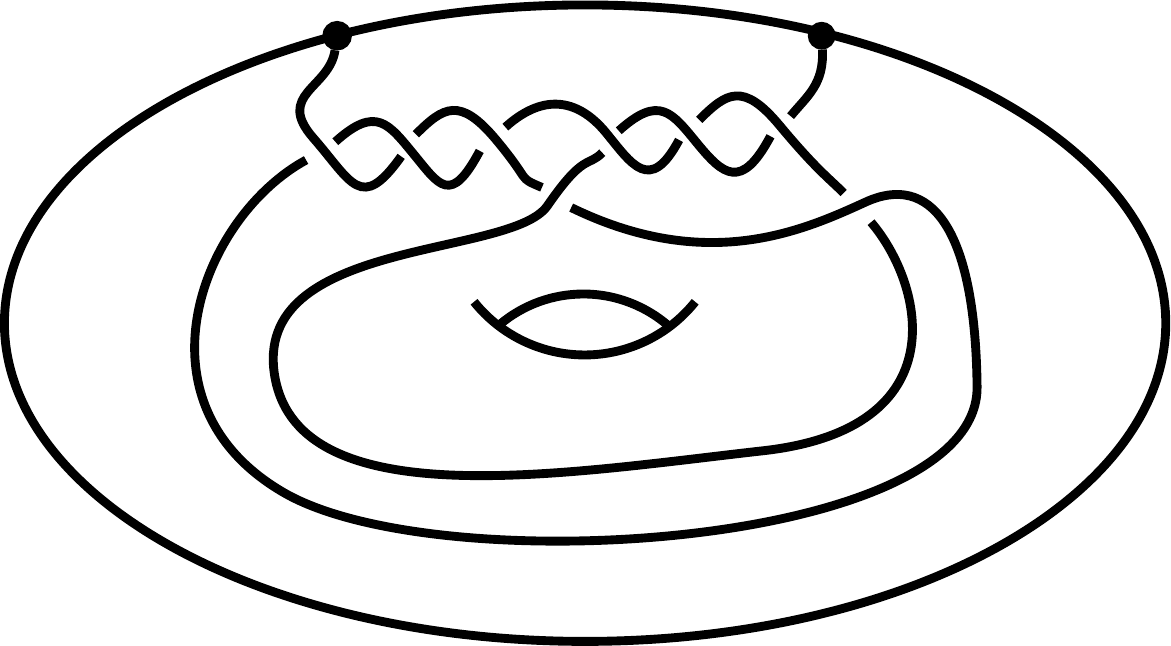}
\caption{A genus-1 tangle, denoted by $\e$.}\label{fig:85ex}
\end{figure}

We have that
$$ \langle \e, g_{i,\vareps}\rangle_D   =  \begin{minipage}{1.15in}\includegraphics[width=1.1in]{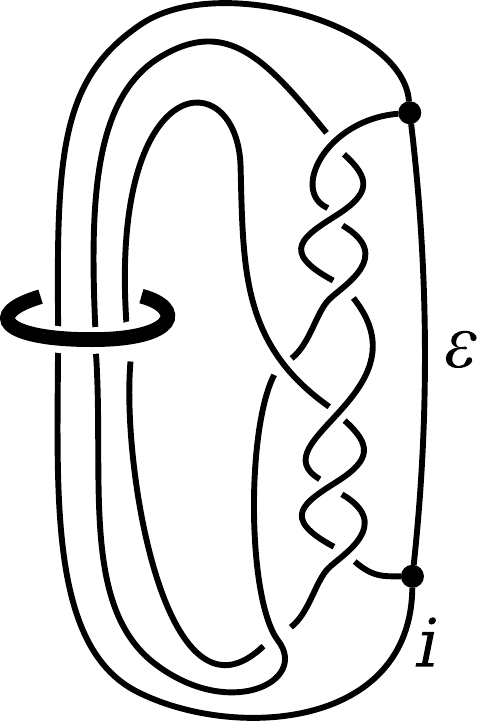}\end{minipage}  \quad \text{is the sum of } $$
\noindent $
 \frac{  \lambda^{1 \text{ }1}_i (\lambda^{1\text{ }1}_j)^{-3} (\lambda^{1\text{ }1}_k)^{-3} 
 \lambda^{1\text{ }1}_l  
 \Delta_j  \Delta_k  \Delta_l
 \Tet \begin{bmatrix} 1 & i & \vareps\\ 1 & j & 1  \end{bmatrix}
  \Tet  \begin{bmatrix} \vareps & i & 1 \\ 1 & k & 1  \end{bmatrix}
  \Tet \begin{bmatrix} 1 & 1 & l \\ 1 & \vareps & j  \end{bmatrix}
    \Tet \begin{bmatrix} 1 & k & \vareps \\ 1 & l & 1  \end{bmatrix}}
 {\theta(1, 1, i) \theta(1,1,j) \theta(1,1,k) \theta(1,1,l) \theta(1,j,\vareps) \theta(1, k, \vareps)  \theta(1,l,\vareps)}
 $
 
\vspace{10pt}

\noindent 
over all $j$, $k$, and $l$ such that the following triples are admissible: $(1,1,i)$, $(1,1,j)$, $(1,j,\vareps)$, $(1,1,k)$, $(1,k, \varepsilon)$,   $(1,1,l)$, and
$(1,l,\varepsilon)$. Admissibility conditions imply that 0 and 2 are the only possible admissible values for $j$, $k$, and $l$.  So, $\langle \e, g_{i,\vareps}\rangle_D =0$ unless $i=0$ or $i=2$.

The coefficients for $\e$ as a linear combination of the graph basis are 
\begin{align*}c_{0,1} =& \frac{-1+2A^4-3A^8+2A^{12}-3A^{16}+2A^{20}-A^{24}+A^{28}}{A^{12}+A^{16}} \\c_{2,1} = & \frac{-1+A^4-2A^8+3A^{12}-2A^{16}+3A^{20}-2A^{24}+A^{28}}{A^{18}+A^{22}+A^{26}} \text{ and  } c_{2,3} = A^4.\end{align*} 

The following generating sets are obtained after further computation.

\begin{proposition}
The even Kauffman bracket ideal of $\e$ is $I_\e^\text{even} = \langle 5,1+A^4\rangle$ which is non-trivial. The odd Kauffman bracket ideal of $\e$ is  $I_\e^\text{odd} = \langle 9,4+A^4\rangle$ which is also non-trivial.
\end{proposition}

These corollaries follow immediately.

\begin{corollary}
The Kauffman bracket ideal $I_\e$ of the genus-1 tangle $\e$ is trivial.
\end{corollary}

\begin{corollary}
The genus-1 tangle $\e$ does not embed in the unknot.
\end{corollary}

Although $\e$ is not obstructed from embedding in the unknot by the ordinary Kauffman bracket ideal, the even and odd Kauffman bracket ideals, working together, do provide an obstruction.


\section{Relation to Determinants}\label{section:determinant}

Let $\omega$  denote $e^{\frac{\pi i}4}$, and $\Omega:\laurent \rightarrow \mathbb{Z}[\omega]$ be the ring epimorphism sending $A$ to $\omega$.
According to  \cite[Prop. 1 on p. 329; \S 11]{kr}, 
\begin{equation}\det(L)=  \omega^j \Omega({\langle L \rangle'}) \label{kre}\end{equation} for an integer $j$, chosen so that $\omega^j \Omega({\langle L \rangle'})$ is a non-negative integer. In fact, \ref{kre} follows easily from  \cite[Corollary 3]{J}
and \cite[Thm 2.8]{K} without consideration of the ``monocyclic states'' used in \cite{kr}. 

\begin{proposition}\label{omeg}
If $L$ is a closure of  tangle $\G$, then $\det(L) \in \Omega(I_\G) \cap   \mathbb{Z}$.
If $L$ is an even closure, then $\det(L) \in \Omega(\ev{\G}) \cap   \mathbb{Z}.$
If $L$ is an odd  closure, then $\det(L) \in \Omega(\od{\G}) \cap   \mathbb{Z}.$
\end{proposition}

\begin{proof} If $L$ 
is
 a closure of  $\G$, $A^j \langle L \rangle' \in I_\G$ for all $j \in \mathbb{Z}$. So for some $j$,
 $$\det(L)= \omega^j \Omega (\langle L \rangle') =\Omega (A^j \langle L \rangle') \in \Omega(I_\G).$$
 As $\det(L) \in \mathbb{Z}$, $\det(L) \in  \Omega(I_\G) \cap \mathbb{Z}$. The other statements are proved similarly.
\end{proof}

Let $\langle n \rangle _\mathbb{Z}$
 denote the $\mathbb{Z}$-ideal generated by $n$.
For the ideals computed in the examples above, noting that $\Omega(A^4)=-1$, we have: 

\begin{equation} \Omega(\od{\A}) \cap \mathbb{Z} = \Omega(\langle 9,4+A^4 \rangle) \cap \mathbb{Z} =\Omega(\langle 3 \rangle) \cap \mathbb{Z}=\langle 3 \rangle_\mathbb{Z}.\label{3}\end{equation}
\begin{equation*}\Omega(\ev{\e}) \cap \mathbb{Z}= \Omega(\langle 9,4+A^4 \rangle) \cap \mathbb{Z}=\Omega(\langle 3 \rangle) \cap \mathbb{Z}=\langle 3 \rangle_\mathbb{Z}.\end{equation*}
\begin{equation*} \Omega(\ev{\D}) \cap \mathbb{Z}= \Omega(\langle 9,-2+A^4 \rangle) \cap \mathbb{Z}=\Omega(\langle 3 \rangle) \cap \mathbb{Z}= \langle 3 \rangle_\mathbb{Z}.\end{equation*}
\begin{equation*} \Omega(\od{\e}) \cap \mathbb{Z}= \Omega(\langle 5,1+A^4 \rangle) \cap \mathbb{Z}=\Omega(\langle 5 \rangle) \cap \mathbb{Z}= \langle 5 \rangle_\mathbb{Z}.\end{equation*}

 The first sentence in  
Proposition \ref{omeg2} 
has the same content as  \cite[Theorem 1.3]{ab}.  

\begin{proposition}\label{omeg2}
Let $L$ 
be
 an odd closure of  $\A$, then $\det(L) \equiv 0 \pmod{3}$.
Let $L$ 
be
 an even closure of  $\D$, then $\det(L) \equiv 0 \pmod{3}$.
If $L$ is an odd closure of $\e$, $\det(L) \equiv 0 \pmod{5}$.
If $L$ is an even closure of $\e$, $\det(L) \equiv 0 \pmod{3}$.
\end{proposition}

\begin{proof}
If
$L$ 
is
 an odd closure of  $\A$, by Proposition \ref{omeg}, $\det(L) \in \Omega(\od{\A}) \cap \mathbb{Z}$.
By (\ref{3}), $\det(L) \equiv 0 \pmod{3}$. The other statements are proved similarly.
\end{proof}
 
Note the tangle $\f$ (pictured below)
  was shown in \cite{abkb}
to have $I_\f= \langle 11,4-A^4 \rangle$.   Thus 
$\Omega (I_\f) \cap Z= \langle 11,5 \rangle_\mathbb{Z} =\langle1\rangle_\mathbb{Z}
=\mathbb{Z}$.

\begin{figure}[h]
\includegraphics[height=1.4in]{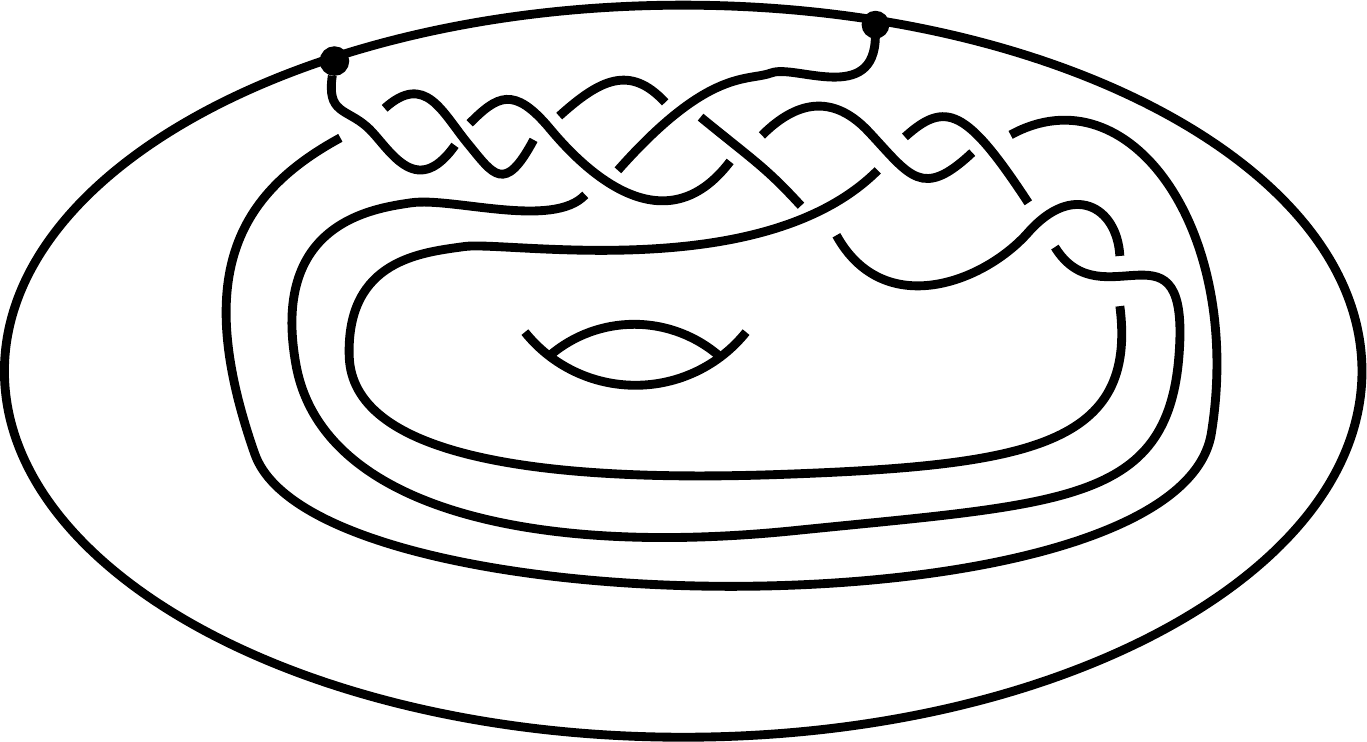}
\caption{The genus-1 tangle $\f$.}\label{fig:tanglef}
\end{figure}

\end{document}